\documentclass{tran-l_for_arXiv}

\usepackage{amssymb, amsmath}
\usepackage{graphics}
\usepackage{graphicx}
\usepackage{eepic}
\usepackage{epic}
\usepackage[all]{xy}
\usepackage{color}
\usepackage{tikz}
\usepackage{url}

\newtheorem{thm}{Theorem}[section]
\newtheorem{lem}[thm]{Lemma}

\newtheorem{prop}[thm]{Proposition}
\newtheorem{ex}[thm]{Example}

\newcommand{\N}{{\mathbb{N}}}
\newcommand{\R}{{\mathbb{R}}}

\newcommand{\Z}{{\mathbb{Z}}}

\newcommand{\cS}{{\mathcal S}}

\newcommand{\cX}{{\mathcal X}}

\def\vector#1{\mbox{\boldmath $#1$}}

\newcommand{\Span}{{\rm Span}}
\newcommand{\Hom}{{\rm Hom}}

\newcommand{\rank}{\mathop{\rm rank}}
\newcommand{\image}{\mathop{\rm im}}
\newcommand{\kernel}{\mathop{\rm ker}}

\newcommand{\positivereal}{\mathbb{R}_{\geq 0}}

\def\positivereal{\mathbb{R}_{\ge 0}}

\numberwithin{equation}{section}
\makeatletter
\def\filename{\texttt{\jobname.tex}} 
\makeatother

\def\E{{\mathbb E}} 
 
\def\la{\lambda}

\newcommand{\wt}{{\rm wt}}

\newcommand{\CL}{{\rm CL}}

\newcommand\cref[1]{Corollary~\ref{#1}}

\begin{document}

\title{Tutte polynomials and random-cluster models in Bernoulli cell complexes}


\author{Yasuaki Hiraoka}
\address{WPI-AIMR, Tohoku University, Sendai, 980-8577, Japan.}
\curraddr{}
\email{hiraoka@wpi-aimr.tohoku.ac.jp}
\thanks{}

\author{Tomoyuki Shirai}
\address{Institute of Mathematics for Industry, Kyushu University,
744, Motooka, Nishi-ku, Fukuoka, 819-0395, Japan}
\curraddr{}
\email{shirai@imi.kyushu-u.ac.jp}
\thanks{}



\date{}

\dedicatory{}

\begin{abstract}
 This paper studies Bernoulli cell complexes from the
 perspective of persistent homology, Tutte polynomials, and random-cluster
 models. Following the previous work \cite{hs}, we first show the asymptotic
 order of the 
 expected lifetime sum of the persistent homology for the Bernoulli cell
 complex process on the $\ell$-cubical lattice. Then, an explicit formula of the
 expected lifetime sum using the Tutte polynomial is derived. Furthermore, we
 study a higher dimensional generalization of the random-cluster model derived from 
 the Edwards-Sokal type coupling, and
 show some basic results such as the positive association and the relation to
 the Tutte polynomial.
\end{abstract}

\maketitle


%
%
%
%
%
\section{Introduction}

Random graphs have been studied in many fields of science such as communication
systems, neural networks, infectious diseases and so on. As a mathematical
framework of random graphs, one of the most standard models is the
Erd\"os-R\'enyi random graphs \cite{er}. Given a complete graph $K_n=(V_n,E_n)$,
the Erd\"os-R\'enyi random graph $G(n,p)$
is defined as a subgraph of $K_n$ with the same
vertex set $V_n$
in such a way that each edge appears in probability $p$. Namely,
this random graph models the connection between each pair of individuals (e.g.,
humans, neurons, etc) independently of the others with the same randomness
parameter.

Recently, the concept of random topology has emerged for the study of higher
dimensional generalizations of random graphs, and is used for studying 
multi-individuals interactions (e.g., \cite{chad}). In this area, the randomness
is often added on the simplicial or cell complexes, and  classical results on
random graphs which can be expressed by 0- or 1-dimensional homology (components
or cycles, respectively) are now being generalized using higher
dimensional homology and also new phenomena have been found. (e.g.,
\cite{kahle,lm}). 

In the previous work \cite{hs}, the authors study the $\ell$-Linial-Meshulam
random process on the maximal simplicial complex with $n$ vertices by using
persistent homology \cite{elz}.
The Linial-Meshulam random process is a natural generalization of 
the Erd\"os-R\'enyi random graph process based on maximal (complete)
simplicial complexes. 
For the $\ell$-Linial-Meshulam process,
they obtain the asymptotic order of the expected
lifetime sum of the persistent homology as $n$ goes to infinity, which can be
regarded as a generalization of Frieze's theorem \cite{frieze}.

In the present paper, we introduce a wider class of random cell complexes
which includes the Bernoulli bond percolation on graphs as well as
the Linial-Meshulam random complex. 
Given a cell complex $X$,
the $\ell$-Bernoulli cell complex process
$\cX=\{X(t) : t\in[0,1]\}$ on $X$ is defined in such a way that we first assign
a uniform random variable $t_\sigma$ on $[0,1]$ for each
$\ell$-cell $\sigma$ independently
and construct a filtration by $X(t)=X^{\ell-1}\sqcup\{\sigma :
t_\sigma\leq t\}$, where $X^{\ell-1}$ is the $(\ell-1)$-skeleton of $X$.
Thus, by fixing $t$, we obtain a random complex $X(t)$ called the Bernoulli
cell complex, which is a higher dimensional generalization of the 
Erd\"os-R\'enyi random graph on $K_n$ and the Bernoulli bond percolation model
on graphs.

Following the previous work, in this paper, we study the Bernoulli
cellular models 
by using persistent homology and Tutte polynomials. First, as in the previous 
work, we derive the order of the expected lifetime sum on the $\ell$-cubical 
lattice, which is the most natural generalization of bond percolation on a
sublattice in $\Z^d$.
Then, we modify the result by Steele \cite{steele} into our
setting by introducing a generalized version of Tutte polynomial
and obtain an explicit formula of the lifetime sum for arbitrary cell
complexes.  
Furthermore, we also investigate a higher dimensional generalization of the
random-cluster model \cite{grimmett} for the Bernoulli cell complex.
The random-cluster model is known to be a variant of the Erd\"os-R\'enyi random
graph in which the edge probabilities are modified respecting the global 
topology. This model has a strong connection to the Ising and Potts models in
statistical mechanics. In this paper, we reconsider the Edwards-Sokal coupling 
\cite{es} from the viewpoint of cohomology,
and generalize it to connect between
the $\ell$-dimensional random-cluster model and the Potts model. 

\section{Persistent homology}
In this paper, $\R_{\geq 0}$ (resp. $\Z_{\geq 0}$) is the set of nonnegative reals (resp. integers).
Let $X$ be a cell complex. 
The set of $\ell$-cells and the $\ell$-skeleton of $X$ are denoted by $X_\ell=\{\sigma^\ell_i : i=1,\dots,n_\ell\}$ and $X^{\ell}$, respectively. For an $\ell$-cell $\sigma^\ell_i$, $\ell$ is called its dimension, and the dimension $\dim X$ of $X$ is given by the maximum dimension of cells in $X$. All cell complexes studied in this paper are assumed to be finite in the sense that $\dim X<\infty$ and $|X_\ell|<\infty$.
Homology and cohomology groups are considered in the cellular setting, 
and its coefficient ring is taken from a field $K$ with characteristic zero, unless specified otherwise.
We refer the reader to \cite{hatcher} for the details of cellular homology and cohomology theory, or refer to Appendix \ref{sec:cell_homology} for its brief exposition.

Let $\cX=\{X(t) :  t\in \positivereal\}$ be a 
right continuous filtration\footnote{In this paper, the term ``filtration'' is used to
mean an increasing sequence of cell complexes as usual in topology.}
of a cell complex $X$. 
Namely, $X(t)$ is a
subcomplex of $X$, $X(t)\subset X(t')$ for $t\leq t'$, and 
$X(t)=\bigcap_{t<t'}X(t')$.
We assume that there exists a saturation time $T$ such that $X(T)=X$.
For each cell $\sigma\in X$, let
$t_{\sigma}=\min\{t\in\positivereal :  \sigma\in X(t)\}$ denote the
birth time of $\sigma$.

Let $K[\R_{\geq 0}]$ be a monoid ring. The elements in $K[\positivereal]$ are expressed by linear combinations of (formal) monomials $az^t$, where $a\in K$, $t\in \positivereal$, and $z$ is an indeterminate. The product of two elements are given by the linear extension of $az^{t}\cdot bz^{s}=abz^{t+s}$.

For a filtration $\cX=\{X(t)\}_{t\in\R_{\geq 0}}$, the {\em persistent homology} of $\cX$ is defined as a graded module
\begin{align}\label{eq:ph}
H_\ell(\cX)=\bigoplus_{t\in\R_{\geq 0}}H_\ell(X(t))
\end{align}
over the monoid ring $K[\R_{\geq 0}]$ with the action 
\begin{align*}
z^a\cdot([c_t])_{t\in\R_{\geq 0}}=([c'_t])_{t\in\R_{\geq 0}},\quad
c'_t=\left\{\begin{array}{ll}
c_{t-a},&\quad t\geq a\\
0,&\quad {\rm otherwise}
\end{array}
\right..
\end{align*}

The persistent homology characterizes the persistence of topological features during the filtration $\cX$. 
In particular, since $\cX$ is defined on a finite cell complex with a saturation time $T$, the structure theorem of the persistent homology holds:
\begin{thm}[\cite{zc}]
There uniquely exist indices $p,q\in\Z_{\geq 0}$ and 
$(b_i,d_i)\in\positivereal^2$ for $i=1,\dots,p$ with $b_i<d_i$ and $b_i\in\positivereal$ for $i=p+1,\dots,p+q$ such that the following isomorphism holds:
\begin{equation}\label{eq:decomposition}
	H_\ell(\cX)\simeq\bigoplus_{i=1}^{p}
	\left((z^{b_i})\biggl/(z^{d_i})
	\right)\oplus\bigoplus_{i=p+1}^{p+q}(z^{b_i}),
\end{equation}
where $(z^a)$ expresses an ideal in $K[\positivereal]$ generated by the monomial $z^a$. 
When $p$ or $q$ is zero, the corresponding direct sum is ignored.
\end{thm}

Here $b_i$ and $d_i$ are called the birth and death times,
respectively, and they measure the events of appearance and
disappearance of topological features in the filtration $\cX$.  
Namely, it expresses that a homology generator is born at $H_\ell(X(b_i))$, persists during $H_\ell(X(t))$ for $b_i\leq t< d_i$, and dies at $H_\ell(X(d_i))$. 
The reduced persistent homology $\tilde H_\ell(\cX)$ is defined by using the reduced homology $\tilde H_\ell(X(t))$ in (\ref{eq:ph}), which corresponds to remove the generator with maximum lifetime in $H_0(\cX)$. 

The lifetime $l_i$ of the birth-death pair $(b_i,d_i)$ is defined by $l_i=d_i-b_i$.
For $p+1\leq i \leq p+q$, we assign the death time
as the saturation time $d_i=T$. 
In this paper, we study the lifetime sum of the $\ell$-th reduced persistent homology $\tilde H_\ell(\cX)$ given by 
$L_\ell=\sum_{i=1}^{p+q} l_i$ for $\ell\geq 1$ and $L_0=\sum_{i=1}^{p+q-1} l_i$. Then, we can easily check the following
\begin{align}\label{eq:lifetime}
L_{\ell}=\int_{[0,T]}\tilde\beta_\ell(t)dt,
\end{align}
where $\tilde\beta_\ell(t)=\tilde\beta_\ell(X(t))=\rank \tilde H_\ell(X(t))$ is the $\ell$-th reduced betti number of $X(t)$.

\section{Bernoulli model for cell complexes}
\subsection{Lifetime sum and generalization of Frieze's theorem}
Let $X$ be a cell complex. 
For $F\subset X_{\ell}$, we set a subcomplex of X by $X_F=X^{\ell-1}\sqcup F$. 
For $p\in[0,1]$,
The \textit{$\ell$-Bernoulli cell complex on $X$} is defined
as a random cell complex whose law is given 
by the probability measure $P_p$ on $\Omega_\ell(X)=\{X_F : F\subset
X_\ell\}$ such that $P_p(X_F)=p^{|F|}(1-p)^{|X_\ell\setminus F|}$. 
We often identify $\Omega_\ell(X)$ with $\{0,1\}^{X_\ell}$ as usual,
and for an element $F\in\{0,1\}^{X_\ell}$, we use  the notation 
\begin{align*}
F(\sigma)=\left\{\begin{array}{cc}
1, &\quad \sigma\in F\\
0, &\quad \sigma\notin F\end{array}
\right..
\end{align*}
Note that, when $X$ is given by the complete graph and $\ell=1$, this model is nothing but the Erd{\"o}s-R{\'e}nyi random graph model \cite{er}.

We next introduce a random filtration of the Bernoulli cell complex model. Let $\{t_\sigma : \sigma\in X_\ell\}$ be i.i.d. random variables uniformly distributed on $[0,1]$. We regard $t_\sigma$ as the birth time of the $\ell$-cell $\sigma$. Let $\cX=\{X^\ell(t)\}_{0\leq t\leq 1}$ be an increasing stochastic process on cell complexes defined by
\begin{align*}
X^\ell(t)=X^{\ell-1}\sqcup\{\sigma\in X_\ell : t_\sigma\leq t\}.
\end{align*}
The process starts from the $(\ell-1)$-skeleton $X^{\ell-1}$ at time $0$ and ends up with the $\ell$-skeleton $X^\ell$ at time 1, i.e.,
\begin{align*}
X^{\ell-1}=X^\ell(0)\subset X^\ell(t)\subset X^\ell(1)=X^\ell.
\end{align*}
We call $\cX$ the \textit{$\ell$-Bernoulli cell complex process on $X$}.
We note that, by definition, $X^\ell(t)$ is equal in law to
the $\ell$-Bernoulli cell complex for each $t\in[0,1]$. 

In the paper \cite{hs}, the authors study the $\ell$-Linial-Meshulam process
($\ell\geq 1$), which is nothing but
the Bernoulli cell complex process on the maximal simplicial complex $\Delta_n$ with $n$-vertices, and show the asymptotic order of the
expected lifetime sum of the $(\ell-1)$-st reduced persistent homology.  
\begin{thm}[\cite{hs}]\label{thm:lifetime1}
Let $L_{\ell-1}$ be the lifetime sum of the $(\ell-1)$-st reduced persistent homology of the $\ell$-Linial-Meshulam process  {\rm ($\ell\geq 1$)} on $\Delta_n$. Then,
\begin{align*}
\E[L_{\ell-1}]=O(n^{\ell-1})
\end{align*}
as $n\rightarrow\infty$.
\end{thm}

One of the key steps to obtain this order is the following formula of the lifetime sum relating to the weight of  spanning acycles. We recall \cite{hs} that a subset $S\subset X_k$ is called a \textit{$k$-spanning acycle} of $X$ if $\tilde H_k(X_S;\Z)=0$ and $|\tilde H_{k-1}(X_S;\Z)|<\infty$, and denote the set of $k$-spanning acycles by $\cS^k$.  The weight of $S$ is given by $\wt(S)=\sum_{\sigma\in S}t_\sigma$.

\begin{thm}[\cite{hs}]\label{thm:lifetime2}
Let $\cX=\{X(t)\}_{t\in \R_{\geq 0}}$ be a filtration of a  cell complex $X$ satisfying 
\begin{align*}
\tilde\beta_{\ell-1}(X^\ell)=\tilde\beta_{\ell-2}(X^{\ell-1})=0. 
\end{align*}
Then, the lifetime sum of the $(\ell-1)$-st reduced persistent homology of $\cX$ is expressed as
\begin{align*}
L_{\ell-1}=\min_{S\in \cS^\ell}\wt (S) - \max_{S\in\cS^{\ell-1}}\wt(X_{\ell-1}\setminus S).
\end{align*}
\end{thm}
Here, we remark that the original formula is derived for simplicial complexes, 
and all the proofs to derive the above formula for cell complexes are similarly performed as in the original discussion. 
We also note from Theorem \ref{thm:lifetime2} that $L_{0}$ is the same
as the weight of the minimum spanning tree. Hence, Theorem
\ref{thm:lifetime1} can be regarded as a higher dimensional
generalization of Frieze's theorem \cite{frieze}, known as $\E[L_0] \to
\zeta(3)$ as the number of vertices tends to $\infty$,
where $\zeta(s)$ is Riemann's zeta function.

%
%
%
%
%
\subsection{Bernoulli cell complex on the $\ell$-cubical lattice}

Let $I$ be a closed interval in $\R$ of the form $I=[a,a+1]$ or $I=[a,a]$ for some $a\in\Z$. We call these intervals elementary intervals. Elementary intervals of the form $[a,a+1]$ (resp. $[a,a]$) are called nondegenerate (resp. degenerate). A cell 
\begin{align*}
Q=I_1\times \cdots\times I_\ell\subset \R^\ell
\end{align*}
consisting of elementary intervals $I_k$, $k=1,\dots,\ell$, is called an elementary cube in $\R^\ell$. The dimension $\dim Q$ of $Q$ is given by the number of nondegenerate intervals. For more details on cubical settings, we refer the reader to \cite{kmm}.

For $\vector{n}=(n_0,\dots,n_\ell)$ with $n_k\in \N$,
let $\widetilde{\CL}(\vector{n})$ be the cell complex consisting of all the
elementary cubes $Q$ in $[0,n_0]\times \cdots \times [0,n_\ell]\subset \R^{\ell+1}$.
We define the {\em $\ell$-cubical lattice} $\CL(\vector{n})$ in $\R^{\ell+1}$ as the
$\ell$-skeleton of $\widetilde{\CL}(\vector{n})$. 
For $\vector{n}=(n,\dots,n)$ with the same entry $n$,
we simply denote them by $\widetilde{\CL}(n)$ and $\CL(n)$, respectively. 

When $\ell=1$, the $1$-Bernoulli cell complex model on $\CL(n)$
in $\R^2$ is the Bernoulli bond percolation model on a sublattice in $\Z^2$. 
In this section, we prove the following theorem.
\begin{thm} \label{thm:lifetime_cl}
 Let $L_{\ell-1}$ be
 the expected lifetime sum of the $\ell$-Bernoulli cell complex process on the
 $\ell$-cubical lattice $\CL(n)$ in $\R^{\ell+1}$. Then,
\[
 \E[L_{\ell-1}] = O(n^{\ell+1})
\]
as $n\rightarrow\infty$.
\end{thm}
This theorem shows that the asymptotic order of
the expected lifetime sum is equal to that of
the number of vertices in $\CL(n)$. 

Before proving the theorem, we give some basic properties of
$\CL(\vector{n})$ and $\ell$-spanning acycles on it. 


\begin{lem} Let $\tilde X = \widetilde{\CL}(\vector{n})$. Then, for $k \in \{0,1,\dots,
 \ell+1\}$, 
\begin{equation}
|\tilde X_k| = \sum_{p=k}^{\ell+1} {p \choose k} S_p(\vector{n}), 
 \label{numberoffaces}
\end{equation}
 where $S_p(\vector{n})$ is the elementary symmetric polynomial
 in $\vector{n}$ of degree $p$. 
\end{lem}
\begin{proof}
Let $\tilde{X}=\widetilde{\CL}(\vector{n})$. 
We can see that 
\begin{align*}
|\tilde{X}_k|=\sum_{\substack{I\subset[\ell]\\|I|=k}}\left(\prod_{i\in
 I}n_i\right)\left(\prod_{j\in [\ell]\setminus I}(n_j+1)\right) 
\end{align*}
 for $k=0,1,2,\dots, \ell+1$, where $[\ell]=\{0,\dots,\ell\}$. 
It is also easy to see that 
\begin{align*}
 G(z,\vector{n}):=\sum_{k=0}^{\ell+1}|\tilde{X}_k|z^k
 =\prod_{i=0}^{\ell}(n_iz+n_i+1). 
\end{align*}
By expanding the right-hand side, we have 
\begin{align*}
 G(z,\vector{n})
 = \sum_{p=0}^{\ell+1} S_{p}(\vector{n}) (1+z)^p
 = \sum_{k=0}^{\ell+1} z^k \sum_{p=k}^{\ell+1}
 {p \choose k} S_{p}(\vector{n}), 
\end{align*}
and this completes the proof.
\end{proof}
We remark that, since $X=\CL(\vector{n})$ is the $\ell$-skeleton of $\tilde X = \widetilde{\CL}(\vector{n})$, $|X_k|=|\tilde X_k|$ for $k\in\{0,\dots,\ell\}$.

\begin{prop}\label{prop:face_number_sa}
For $\vector{n}=(n_0,\dots,n_\ell)$, the number of $\ell$-cells in $\ell$-spanning acycles is 
$N(\vector{n})=\ell S_{\ell+1}(\vector{n})+S_\ell(\vector{n})$.
\end{prop}
\begin{proof}
 Let $\tilde{X}=\widetilde{\CL}(\vector{n})$, $X=\CL(\vector{n})$
 and $F\subset X_\ell$.
 The Euler characteristics of $\tilde{X}$ and $X_F$
 are given by $\chi(\tilde{X}) = 1$
 and $\chi(X_F) = 1 + (-1)^{\ell-1} \tilde\beta_{\ell-1}(X_F) 
 + (-1)^{\ell} \tilde\beta_{\ell}(X_F)$, respectively, 
 since $\tilde{X}$ is contractible and
 the $(\ell-1)$-skeleton of $X_F$ is homotopy equivalent to the wedge sum of
 $(\ell-1)$-spheres.
 By applying the Euler-Poincar\'{e} formula
 to $\tilde{X}$ and $X_F$ and taking the difference,
 we easily see that 
 $\tilde\beta_\ell(X_F)$ and $\tilde\beta_{\ell-1}(X_F)$ are related as 
\begin{align}\label{eq:rho}
 \tilde\beta_\ell(X_F)=\tilde\beta_{\ell-1}(X_F)+|F|-
 (|\tilde{X}_{\ell}| - |\tilde{X}_{\ell+1}|). 
\end{align}
 Since $\tilde\beta_\ell(X_F)=\tilde\beta_{\ell-1}(X_F) = 0$ for $F \in \cS^\ell$, 
(\ref{numberoffaces}) and (\ref{eq:rho}) lead to
 $N(\vector{n}) = |F|
=  |\tilde{X}_{\ell}| - |\tilde{X}_{\ell+1}|
 = \ell S_{\ell+1}(\vector{n})+
 S_{\ell}(\vector{n})$ for $F\in\cS^\ell$.
%
\end{proof}

 \begin{proof}[Proof of Theorem~\ref{thm:lifetime_cl}]
  From (\ref{numberoffaces}) the inequality $\tilde\beta_{\ell-1}(t)\leq
  |X_{\ell-1}| = {\ell+1\choose 
  2}n^{\ell+1}+O(n^\ell)$ holds, which together with
  (\ref{eq:lifetime}) implies
  $\E[L_{\ell-1}] \le {\ell+1\choose 2}n^{\ell+1}+O(n^\ell)$. 
For lower bound, we note that the assumption in Theorem
\ref{thm:lifetime2} is satisfied on $\CL(n)$, and we have
$L_{\ell-1}=\min_{S\in \cS^\ell}\wt(S)$. Let $0\leq
t_{\sigma_1}<\dots<t_{\sigma_{m}}\leq 1$ be the reordering of the
  $\ell$-cells with respect to the birth times, where
  $m=|X_{\ell}|=(\ell+1)n^\ell(n+1)$ from (\ref{numberoffaces}). 
Hence, Theorem \ref{thm:lifetime2}, Proposition
\ref{prop:face_number_sa}, and the expectation of the ordered statistics
lead to 
\begin{align*}
\E[L_{\ell-1}]\geq \E[t_{\sigma_1}+\dots+t_{\sigma_{N(n)}}]=\sum_{k=1}^{N(n)}\frac{k}{m+1}=\frac{\ell^2}{2(\ell+1)}n^{\ell+1}+O(n^\ell).
\end{align*}
This concludes the proof. 
\end{proof}

\section{Tutte polynomial and expected lifetime sum}
This section derives an explicit formula for the expected lifetime sum in the
Bernoulli cell complex process using the Tutte polynomial. In this paper, we define
the \textit{$\ell$-Tutte polynomial} of a cell complex $X$ as
\begin{align}\label{eq:tutte}
T_\ell(X;x,y)=\sum_{F\subset X_\ell}(x-1)^{\tilde\beta_{\ell-1}(X_F)}(y-1)^{\tilde\beta_\ell(X_F)},
\end{align}
This definition is essentially the same as \cite{kr}, and its contraction-deletion reduction is also studied in \cite{bbc}. We also note that $T_\ell(X;1,1)=|\cS^\ell|$ counts the number of $\ell$-spanning acycles.

First of all, as used in (\ref{eq:rho}), we have
\begin{align*}
\tilde\beta_\ell(X_F)=\tilde\beta_{\ell-1}(X_F)+|F|-\rho(X),
\end{align*}
where $\rho(X)$ is independent of the choice of $F\subset X_\ell$ and expressed as
\begin{align*}
\rho(X)=(-1)^\ell\left(\sum_{k=0}^{\ell-2}(-1)^k\tilde\beta_k(X)-\sum_{k=0}^{\ell-1}(-1)^k|X_k|+1\right).
\end{align*}
Then, the $\ell$-Tutte polynomial can be represented as an
expectation with respect to the Bernoulli measure: 
\begin{align*}
T_\ell(X;x,y)&=\sum_{F\subset X_\ell}(x-1)^{\tilde\beta_{\ell-1}(X_F)}(y-1)^{\tilde\beta_\ell(X_F)}\\
&=\sum_{F\subset X_\ell}(x-1)^{\tilde\beta_{\ell-1}(X_F)}(y-1)^{\tilde\beta_{\ell-1}(X_F)+|F|-\rho(X)}\\
&=y^{|X_\ell|}(y-1)^{-\rho(X)}\sum_{F\subset X_\ell}\left(1-\frac{1}{y}\right)^{|F|}\left(\frac{1}{y}\right)^{|X_\ell\setminus F|}\{(x-1)(y-1)\}^{\tilde\beta_{\ell-1}(X_F)}\\
&=y^{|X_\ell|}(y-1)^{-\rho(X)}\E\left[\{(x-1)(y-1)\}^{\tilde\beta_{\ell-1}(X_F)}\right]. 
\end{align*}
Here the law of $X_F$ is given by the Bernoulli cell complex model with
probability $1-\frac{1}{y}$. 

For a Bernoulli cell complex process $\cX=\{X(t)\}_{0\leq t\leq 1}$ of $X$,
we consider the Laplace transform of $\beta_{\ell-1}(X(t))$ defined by 
\begin{equation}
\phi(\la,t)=\E\left[e^{\la\tilde\beta_{\ell-1}(X(t))}\right].
\label{eq:laplace1}
\end{equation}
Then, by setting $e^\la=(x-1)(y-1)$ and $t=1-\frac{1}{y}$, we immediately obtain from
(\ref{eq:tutte}) that
\begin{equation}
 \phi(\la,t) =y^{-|X_\ell|}(y-1)^{\rho(X)}T_\ell\left(X;1+\frac{1-t}{t}e^{\la},
\frac{1}{1-t}\right).
\label{eq:laplace2} 
\end{equation}
From these expressions (\ref{eq:laplace1}) and (\ref{eq:laplace2}),
by taking logarithmic derivative at $\la=0$, we have
\begin{align*}
\left.\frac{\partial}{\partial \la}\log
 \phi(\la,t)\right|_{\la=0}
 = \E[\tilde\beta_{\ell-1}(X_t)]
=\frac{1-t}{t}\cdot\frac{\partial_xT_\ell\left(X;\frac{1}{t},\frac{1}{1-t}\right)}{T_\ell\left(X;\frac{1}{t},\frac{1}{1-t}\right)}. 
\end{align*}
Thus, from (\ref{eq:lifetime}), this leads to following formula.
\begin{thm}
Let $L_{\ell-1}$ be the lifetime sum of the $(\ell-1)$-st reduced persistent
  homology of the $\ell$-Bernoulli complex process on $X$. Then,  
\begin{align}\label{eq:expect_lifetime_sum}
\E[L_{\ell-1}]=\int_0^1
 \frac{1-t}{t}\cdot\frac{\partial_xT_\ell\left(X;\frac{1}{t},\frac{1}{1-t}\right)}{T_\ell\left(X;\frac{1}{t},\frac{1}{1-t}\right)}dt. 
\end{align}
\end{thm}

This can be regarded as a formula for minimum spanning
acycle taking Theorem~\ref{thm:lifetime2} into account.  
In this context, the above formula for $\ell=1$ is derived in \cite{steele}. 

\begin{ex}{\rm
Let $\Delta^2_n$ be the 2-skeleton of the maximal simplicial complex with $n$
 vertices. Let us write $L_1(n)$ for the lifetime sum of the 1st reduced persistent
 homology of the 2-Linial-Meshulam process on $\Delta^2_n$.
For $n=4,5$, the Tutte polynomials (\ref{eq:tutte}) and the expected lifetime sum (\ref{eq:expect_lifetime_sum}) are obtained as follows. 
\begin{align*}
&T_2(\Delta^2_4;x,y)=(x-1)^3+4(x-1)^2+6(x-1)+4+(y-1)=x^3+x^2+x+y\\
&\E[L_1(4)]=\int^1_0(1-t)^2(3+2t+t^2)dt=\frac{6}{5}\\
&T_2(\Delta^2_5;x,y)=6x+15x^2+15x^3+10x^4+4x^5+x^6+6y+20xy+15x^2y+5x^3y+11y^2\\
&\hspace{2cm}+10xy^2+6y^3+y^4\\
 &\E[L_1(5)]=\int^1_0(1-t)^3(1+t)(6+2t+4t^2-4t^3-t^4-8t^5+6t^6)dt=\frac{1817}{924}=1.96645\dots
\end{align*}
Similarly, we can compute $\E[L_1(6)] = \frac{5337295}{1939938}=2.75127\dots$. 
 }
\end{ex}
%
%
%
%
%
\section{The $\ell$-random-cluster model}
The $\ell$-Bernoulli cell complex is regarded as the product measure of those defined on $\ell$-cells with the same probability $p$. In the context of the Erd\"os-R\'enyi random graph, there is a variant known as the random-cluster model \cite{grimmett} which differs from the product measure by respecting the topology of the connectedness. For $p\in[0,1]$ and $q>0$, the random-cluster measure $\phi_{p,q}$ on $\Omega_1(X)$ is defined by 
\begin{align}\label{eq:rc1}
\phi_{p,q}(X_F)=\frac{1}{Z_{\rm RC}}p^{|F|}(1-p)^{|X_1\setminus F|}q^{\beta_0(X_F)},
\end{align}
where $Z_{{\rm RC}}$ is the normalizing constant (or partition function).
By definition, the Erd\"os-R\'enyi random graph corresponds to the model with $q=1$. 

When $q$ is an integer with $q\geq 2$, the random-cluster model is known to be related to the so-called Ising and Potts models arising in the statistical mechanics. 
In this model, we consider an assignment $s\in S=\{0,1,\dots,q-1\}^{X_0}$ of a value $s_x\in\{0,1,\dots,q-1\}$ to each vertex $x\in X_0$. 
For $s \in S$ and $e=|xy|\in X_1$, let us write
$\delta_e(s)=\delta_{s_x,s_y}$, where $\delta_{a,b}$ is the Kronecker
delta. Then, the probability law of the Potts model (the Ising model for $q=2$) is given by 
\begin{align}\label{eq:potts}
\pi_{\alpha,q}(s)=\frac{1}{Z_{\rm P}}e^{-\alpha H(s)},\quad s\in S,
\end{align}
where $Z_{\rm P}$ is the normalizing constant and the Hamiltonian is given by
\begin{align*}
H(s)=-\sum_{e\in X_1}\delta_e(s).
\end{align*}

Then, it is known in \cite{es} that the random-cluster model and
the Potts model can be coupled with
a coupling measure $\mu$ on $S\times \Omega_1(X)$ defined by 
\begin{align}\label{eq:es_couple}
\mu(s,X_F)=\frac{1}{Z_{\rm ES}} \prod_{e\in
 X_1}\{(1-p)\delta_{F(e),0}+p\delta_{F(e),1}\delta_e(s)\} 
\end{align}
so that the random-cluster model and Potts model are obtained as the marginals of $s$ and $X_F$, respectively.

In the rest of this  section, we modify the coupling
(\ref{eq:es_couple}) so that the higher dimensional generalizations of
the random-cluster model and the Potts model based on the
Bernoulli cell complex model are derived as the marginals.  
The key for the higher dimensional generalization is to regard $s\in S$
as a $0$-cochain in $C^0(X;\Z_q)$ and $\delta_e(s)$ as a local
obstruction of $s$ along the edge $e$. We also note that the coupling
(\ref{eq:es_couple}) can also be written as 
\begin{align*}
\mu(s,X_F)=\frac{1}{Z_{\rm ES}} (1-p)^{|X_1\setminus F|}p^{|F|}\prod_{e\in F}\delta_e(s).
\end{align*}
Then, $\prod_{e\in F}\delta_e(s)=1$ if and only if $\partial^0s|_{X_F}=0$, i.e., $s\in Z^0(X_F;\Z_q)$. Here $\partial^0$ and $Z^0(X_F;\Z_q)$ express the $0$-coboundary map of the cochain complex $C_*(X;\Z_q)$ and the cocycle group of $X_F$. 

From this observation, we now generalize the coupling so that the marginals naturally define a random-cluster model and a Potts model on a cell complex $X$. 
In the following derivation, we assume $\Z_q=\{0,1,\dots,q-1\}$ to be a finite field with the order $q$ and use it for the coefficient ring of cohomology. 
Let 
\begin{align*}
\xymatrix{
\cdots\ar[r]&C^{\ell-1}(X;\Z_q)\ar[r]^{\partial^{\ell-1}}&C^\ell(X;\Z_q)\ar[r]^{\quad\partial^\ell}&\cdots
}
\end{align*}
be the cellular cochain complex and $H^k(X;\Z_q)$ be its cohomology.  We note that $H^k(X;\Z_q)\simeq \Hom_{\Z_q}(H_k(X;\Z_q),\Z_q)$, and especially, $\dim H_k(X;\Z_q)=\dim H^k(X;\Z_q)$. We denote this dimension by $\beta_k(X;\Z_q)$ and call it the $k$-th Betti number with $\Z_q$ coefficient.

We define a coupling $\mu$ on $C^{\ell-1}(X;\Z_q)\times\Omega_\ell(X)$ by
\begin{align}\label{eq:lm_coupling}
	\mu(s,X_F)\propto\prod_{\sigma\in X_\ell}
	\left\{
		(1-p)\delta_{F(\sigma),0}+p\delta_{F(\sigma),1}\delta_{\sigma}(s)\right\},
\end{align}
where
\[
\delta_\sigma(s)=\left\{\begin{array}{ll}
1,&\quad\partial^{\ell-1} s(\sigma)=0\\
0,&\quad{\rm otherwise}
\end{array}\right..
\]
We note that, when $\ell=1$, this is the same as the original definition of $\delta_\sigma$.

The first marginal becomes
\begin{align*}
\sum_{X_F\in\Omega_\ell(X)}\mu(s,X_F)&\propto\sum_{X_F\in\Omega_\ell(X)}\prod_{\sigma\in X_\ell}
	\left\{(1-p)\delta_{F(\sigma),0}+p\delta_{F(\sigma),1}\delta_{\sigma}(s)\right\}\\
	&=\prod_{\sigma\in X_\ell}\{(1-p)+p\delta_\sigma(s)\}.
\end{align*}
By setting $p=1-e^{-\alpha}$, we have
\begin{align*}
\sum_{X_F\in\Omega_\ell(X)}\mu(s,X_F)&\propto\prod_{\sigma\in X_\ell}\{e^{-\alpha}+(1-e^{-\alpha})\delta_\sigma(s)\}\\
&=e^{-\alpha|X_\ell|}\prod_{\sigma\in X_\ell}\{1+(e^\alpha-1)\delta_\sigma(s)\}\\
&=e^{-\alpha|X_\ell|}e^{-\alpha H(s)},
\end{align*}
where $H(s)=-\sum_{\sigma\in X_\ell}\delta_\sigma(s)$. 
For $\ell=1$, this recovers the $q$-Potts model (\ref{eq:potts}).

The second marginal becomes
\begin{align}
\sum_{s\in C^{\ell-1}(X;\Z_q)}\mu(s,X_F)&\propto\sum_{s\in C^{\ell-1}(X;\Z_q)}\prod_{\sigma\in X_\ell}
	\left\{
		(1-p)\delta_{F(\sigma),0}+p\delta_{F(\sigma),1}\delta_{\sigma}(s)\right\}\\
&=(1-p)^{|X_\ell\setminus F|}p^{|F|}\sum_{s\in C^{\ell-1}(X;\Z_q)}\prod_{\sigma\in F}\delta_\sigma(s).\nonumber
\end{align}
Note that $\prod_{\sigma\in F}\delta_\sigma(s)=1$ if and only if $s\in Z^{\ell-1}(X_F;\Z_q)$. Hence, we have
\begin{align}\label{eq:rc_pre}
\sum_{s\in C^{\ell-1}(X)}\mu(s,X_F)\propto(1-p)^{|X_\ell\setminus F|}p^{|F|}q^{\dim Z^{\ell-1}(X_F;\Z_q)}.
\end{align}
For $\ell=1$, since $\beta_0(X_F)=\beta_0(X_F;\Z_q)=\dim Z^0(X_F;\Z_q)$, this recovers the random-cluster model (\ref{eq:rc1}).

%
%
%
%
%
The second marginal (\ref{eq:rc_pre}) is defined on the space $\Omega_\ell(X)$ in which each element $X_F$ has the same $(\ell-1)$-skeleton $X^{\ell-1}$. This leads to $B^{\ell-1}(X_F;\Z_q)=B^{\ell-1}(X;\Z_q)$, and in particular, $\dim B^{\ell-1}(X_F;\Z_q)$ is independent of the choice of $F$. Therefore, by appropriately changing the normalizing constant, the second marginal has the following formulation
\begin{align*}
\mu_{p,q}(Y)=\frac{1}{Z_{p,q}}p^{|Y_\ell|}(1-p)^{|X_\ell\setminus Y_\ell|}q^{\beta_{\ell-1}(Y;\Z_q)},\quad\quad Y\in\Omega_\ell(X).
\end{align*}

In the derivation above, we assume that $q$ is a prime number, and $\beta_{\ell-1}(Y;\Z_q)$ is dependent on the choice of $q$ for $\ell>1$. 
In what follows, we study a slightly generalized probability measure
\begin{align*}
\mu_{p,q}(Y)=\frac{1}{Z_{p,q}}p^{|Y_\ell|}(1-p)^{|X_\ell\setminus Y_\ell|}q^{\beta_{\ell-1}(Y)},\quad\quad Y\in\Omega_\ell(X)
\end{align*}
for $p\in [0,1]$ and $q>0$. Here, we also allow to take the coefficient of the Betti number $\beta_{\ell-1}(Y)=\beta_{\ell-1}(Y;K)$ in some fixed field $K$, which is not necessary to be $\Z_q$. We call this probability measure the {\em $\ell$-random-cluster measure} on a cell complex $X$ with $K$ coefficient. 

In this section, we show two basic properties on $\mu_{p,q}$. Both of
them are independent of the choice of $K$.
\begin{thm}\label{thm:positive_association}
Let $X$ be a cell complex. For $p\in(0,1)$ and $q\geq 1$, the $\ell$-random-cluster measure $\mu_{p,q}$ on $X$ is positively associated, i.e., 
 \begin{align*}
\mu_{p,q}(fg)\geq\mu_{p,q}(f)\mu_{p,q}(g)
 \end{align*}
for any increasing functions $f,g:\Omega_\ell(X)\rightarrow \R$.
\end{thm}
For $\ell=1$, the positive association plays a key role to study phase transitions of the 1-random-cluster model of infinite graphs.  
We generalized the random-cluster model for higher dimension so that
topological nature of this model became clearer. 
Here we give a proof of the theorem 
by emphasizing with a topological viewpoint. 

To prove the theorem, we need the following lemma.
\begin{lem} \label{lem:mayer_vietoris}
For topological spaces $A$ and $B$,
\begin{align*}
\beta_k(A\cap B) + \beta_k(A\cup B)\geq \beta_k(A) +\beta_k(B).
\end{align*}
\end{lem}
\begin{proof}
Let us consider the Mayer-Vietoris sequence:
\begin{align*}
\xymatrix{
\cdots\ar[r]&H_k(A\cap B)\ar[r]^{i\quad}&H_k(A)\oplus H_k(B)\ar[r]^{\quad j}&H_k(A\cup B)\ar[r]^{\delta\quad}&H_{k-1}(A\cap B)\ar[r]&\cdots
}
\end{align*}
This exact sequence leads to the following relations:
\begin{align*}
&\beta_k(A\cap B) = \rank i + \dim\kernel i,\\
&\beta_k(A\cup B) = \rank\delta+\dim\kernel\delta,\\
&\beta_k(A) + \beta_k(B)=\rank j+\dim\kernel j.
\end{align*}
Then, it follows from $\image i=\kernel j$, and $\image j=\kernel\delta$ that
\begin{align*}
\beta_k(A\cap B)+\beta_k(A\cup B)-\beta_k(A)-\beta_k(B)=\dim \kernel i+\rank\delta\geq 0.
\end{align*}
\end{proof}

\begin{proof}[Proof of Theorem \ref{thm:positive_association}]
Since $\mu_{p,q}$ is strictly positive for $p\in(0,1)$, it is sufficient to prove that the measure $\mu_{p,q}$ has the so-called FKG lattice property \cite{grimmett}. Here, the FKG lattice property is expressed as
\begin{align*}
\mu_{p,q}(Y\cup Y')\mu_{p,q}(Y\cap Y')\geq \mu_{p,q}(Y)\mu_{p,q}(Y')
\end{align*}
for $Y,Y'\in\Omega_\ell(X)$. If $q \ge 1$, this is equivalent to show 
\begin{align*}
\beta_{\ell-1}(Y\cup Y')+\beta_{\ell-1}(Y\cap Y')\geq \beta_{\ell-1}(Y)+\beta_{\ell-1}(Y'),
\end{align*}
which is proved from Lemma \ref{lem:mayer_vietoris}.
\end{proof}

Next, we show a relation between the normalizing constant and the $\ell$-Tutte polynomial. Here, for the consistency to the known result for $\ell=1$, we use the non-reduced $\ell$-Tutte polynomial 
\begin{align}\label{eq:non_reduced_tutte}
T_\ell(X;x,y)=\sum_{F\subset X_\ell}(x-1)^{\beta_{\ell-1}(X_F)}(y-1)^{\beta_\ell(X_F)}.
\end{align}
%
 \begin{thm}
Let $X$ be a cell complex. Then, the normalizing constant $Z_{p,q}$ of the $\ell$-random-cluster model on $X$ is expressed as 
\begin{align*}
 Z_{p,q}
 &= \left(\frac{p}{1-p}\right)^{\beta_{\ell-1}(X)} p^{r(X)} (1-p)^{\beta_{\ell}(X^{\ell})}
T_\ell\left(X;1+\frac{q(1-p)}{p},\frac{1}{1-p}\right), 
\end{align*}
where $T(X;x,y)$ is the $\ell$-Tutte polynomial {\rm (\ref{eq:non_reduced_tutte})} and $r(X)=\rank \partial_\ell$.
\end{thm}
\begin{proof}
 Since $\beta_{\ell}(Y) = |Y_{\ell}| - r(Y)$ and
 $\beta_{\ell-1}(Y) = \dim \ker \partial_{\ell-1} - r(Y)$,
 we have 
\[
 |Y_{\ell}| = \beta_{\ell}(Y) - \beta_{\ell-1}(Y)+ \dim \ker \partial_{\ell-1}
\]
 for every $Y \in \Omega_{\ell}$. 
 Therefore,
 \begin{align*}
\lefteqn{\sum_{Y\in \Omega_\ell} p^{|Y_{\ell}|} (1-p)^{|X_{\ell} \setminus
 Y_{\ell}|} q^{\beta_{\ell-1}(Y)}}\\ 
 &= p^{\dim \ker \partial_{\ell-1}} (1-p)^{|X_{\ell}| - \dim \ker
  \partial_{\ell-1}}  
  \sum_{Y\in \Omega_\ell}
  \left(\frac{q(1-p)}{p}\right)^{\beta_{\ell-1}(Y)}
    \left(\frac{p}{1-p}\right)^{\beta_{\ell}(Y)} \\
 &= \left(\frac{p}{1-p}\right)^{\beta_{\ell-1}(X)} p^{r(X)}
  (1-p)^{\beta_{\ell}(X^{\ell})} 
T_\ell\left(X;1+\frac{q(1-p)}{p},\frac{1}{1-p}\right),  
 \end{align*}
\end{proof}

As is the case of the $1$-random-cluster model on graphs, 
the situation is more subtle for $q < 1$. Here we just remark that
 as $p \to 0$ and $q/p \to 0$
 \[
  \mu_{p,q}(Y) \to \frac{1}{|\cS^\ell|} \mathbf{1}(Y \in \cS^\ell), 
 \]
that is, the $\ell$-random-cluster measure converges to the uniform
$\ell$-spanning acycle measure.
%

\section{Discussion}
In this paper, we studied the $\ell$-Bernoulli cell complexes from the
perspective of persistent homology, Tutte polynomials, and the random-cluster
model.  
From the positive association property on the $\ell$-random-cluster
model,
it would become an interesting research area to study infinite cell
complex models 
(e.g., the $\ell$-cubical lattice), thermodynamic limits, and phase transitions.

\section*{Acknowledgement}
This work is partially supported by JSPS Grant-in-Aid (26610025, 26287019) and JST CREST Mathematics (15656429).

\bibliographystyle{amsplain}

\begin{thebibliography}{99}


\bibitem{bbc}
C. Bajo, B. Burdick, S. Chmutov. On the Tutte-Krushkal-Renardy polynomial for cell complexes. arXiv:1204.3563.

\bibitem{chad}
C. Giusti, E. Pastalkova, C. Curto, and V. Itskov. {\it Proceedings of the National Academy of Sciences} 112 (2015), 13455--13460.

%

\bibitem{elz}
	H. Edelsbrunner, D. Letscher, and A. Zomorodian. 
	Topological Persistence and Simplification. 
	{\it Discrete Comput. Geom.} 28 (2002), 511--533.
	
\bibitem{es}
R. G. Edwards and A. D. Sokal. Generalization of the Fortuin-Kasteleyn-Swendsen-Wang representation and Monte Carlo algorithm. The Physical Review D 38 (1988), 2009--2012. 

\bibitem{er}
	P. Erd{\"o}s and A. R{\'e}nyi. 
	On random graphs I. 
	{\it Publ. Math. Debrecen} 6 (1959), 290--297.

\bibitem{frieze}
	A. M. Frieze. 
	On the value of a random minimum spanning tree problem. 
	{\it Discrete Applied Math.} 10 (1985), 47--56.
	
\bibitem{grimmett}
G. Grimmett. The random-cluster model. Springer, 2006.

\bibitem{hatcher}
A. Hatcher. Algebraic Topology. Cambridge University Press, 2001.  

\bibitem{hs}
Y. Hiraoka and T. Shirai. Minimum spanning acycle and lifetime of persistent homology in the Linial-Meshulam process. arXiv:1503.05669.

\bibitem{kmm}
T. Kaczynski, K. Mischaikow, and M. Mrozek. Computational Homology. Springer, 2004.

\bibitem{kahle}
	M. Kahle. 
	Topology of random clique complexes. 
	{\it Discrete Math.} 309 (2009), 1658--1671.
	
\bibitem{kr}
V. Krushkal and D. Renardy. A Polynomial invariant and duality for triangulations. arXiv:1012.1310v4.

\bibitem{lm}
	N.~Linial and R.~Meshulam. 
	Homological connectivity of random $2$-complexes. 
	{\it Combinatorica} 26 (2006), 475--487. 
	
\bibitem{steele}
	J.~M~Steele. Minimal spanning trees for graphs with random edge
	lengths. \\
	\tt{http://stat.wharton.upenn.edu/~steele/Publications/PDF/MSTfGwREL.pdf} 

\bibitem{zc}
	A. Zomorodian and G. Carlsson. 
	Computing persistent homology. 
	{\it Discrete Comput. Geom.} 33 (2005), 249-274.
	
	
\end{thebibliography}

\appendix
 \section{Cellular homology and cohomology}\label{sec:cell_homology}
Let $X=\{\sigma^\ell_i : i=1,\dots,n_\ell, \ell=1,\dots,d\}$ be a $d$-dimensional cell complex. The cellular chain complex of $X$ is defined by the horizontal sequence of the diagram 
\begin{align}\label{eq:cell_chain_hom}
\xymatrix{
&&H_{\ell-1}(X^{\ell-1})\ar[dr]^-{j_{\ell-1}}&&\\
\cdots\ar[r]& H_\ell(X^\ell,X^{\ell-1})\ar[rr]^-{\partial_\ell=j_{\ell-1}\circ d_\ell}\ar[ur]^-{d_\ell} && H_{\ell-1}(X^{\ell-1},X^{\ell-2})\ar[r]&\cdots
}
\end{align}
where the homology is in the sense of singular homology with a field $K$ coefficient (arbitrary characteristics),  $d_\ell$ is the connecting morphism and $j_{\ell-1}$ is induced by the quotient chain map of the singular chain groups. Since $d_{\ell-1}\circ j_{\ell-1}$ is the composition of consecutive maps in the exact sequence of the pair $(X^{\ell-1},X^{\ell-2})$, we have $\partial_{\ell-1}\circ\partial_\ell=0$, showing that the horizontal sequence becomes a chain complex. Then, the cellular homology is defined by $H^{\rm cell}_\ell(X)=\kernel \partial_\ell/\image\partial_{\ell+1}$. 

Because of the isomorphism $H_\ell(X)\simeq H^{\rm cell}_\ell(X)$, we use the same symbol $H_\ell(X)$ even for the cellular homology. 
We also denote the cellular chain complex defined by the horizontal
sequence in (\ref{eq:cell_chain_hom}) by 
\begin{align}\label{eq:cell_chain}
\xymatrix{
\cdots\ar[r]^-{\partial_{\ell+1}} & C_\ell(X)\ar[r]^-{\partial_\ell} & C_{\ell-1}(X)\ar[r]^-{\partial_{\ell-1}} & \cdots
}
\end{align}
as usual.
The reduced homology $\tilde H_\ell(X)$ is defined by $H_0(X)\simeq K\oplus \tilde H_0(X)$ and $H_\ell(X)=\tilde H_\ell(X)$ for $\ell>0$. The betti number (or the reduced betti number, resp.) is given by $\beta(X)=\rank H_\ell(X)$ (or $\tilde\beta(X)=\rank \tilde H_\ell(X)$, resp.).

We recall that $H_\ell(X^\ell,X^{\ell-1})$ is generated by the set $X_\ell$ of $\ell$-cells, i.e.,
\begin{align*}
H_\ell(X^\ell,X^{\ell-1}) \simeq\Span_KX_\ell.
\end{align*}
Thus, we obtain a matrix representation $(M_{i,j})_{1\leq i\leq n_{\ell-1},1\leq j\leq n_{\ell}}$ of $\partial_\ell$ using the bases $X_\ell$ and $X_{\ell-1}$. Here, $M_{i,j}$ is given by the degree of the map 
\begin{align*}
S^{\ell-1}_{\sigma^{\ell}_j}\rightarrow X^{\ell-1}\rightarrow S^{\ell-1}_{\sigma^{\ell-1}_i},
\end{align*}
where the first map is the attaching map of the $\ell$-cell $\sigma^{\ell}_j$ and the latter map is the quotient map collapsing $X^{\ell-1}\setminus \sigma^{\ell-1}_i$ to a point.

Let us consider examples of cell complexes
$X=\{e^0,e^1\}$ and $Y=\{e^0_1,e^0_2,e^1_1,e^1_2\}$ shown in Figure \ref{fig:cell_example}.
Then, the chain complexes of $X$ and $Y$ are given by
\begin{align*}
\xymatrix{
0\ar[r]&K\ar[rr]^{0}&&K\ar[r]&0,\\
0\ar[r]&K^2\ar[rr]^{
\scriptsize{
\left(\begin{array}{cc}
1 & 1\\
-1 & -1
\end{array}\right)
}
}&&K^2\ar[r]&0,
}
\end{align*}
respectively. Hence, we have 
\begin{align*}
H_\ell(X)\simeq H_\ell(Y)\simeq
\left\{
\begin{array}{ll}
K,& \ell=0,1\\
0,& \ell\neq 0,1
\end{array}
\right..
\end{align*}

\begin{figure}[htbp]
 \begin{center}
  \includegraphics[width=0.5\hsize]{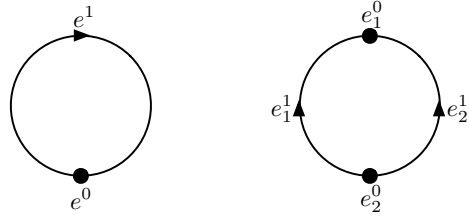}
  \caption{1-dimensional cell complexes $X$ (left) and $Y$ (right). }
  \label{fig:cell_example}
 \end{center}
\end{figure}

By taking the dual $(\bullet)^*=\Hom_K(\bullet,K)$ of (\ref{eq:cell_chain}), we obtain a cochain complex
\begin{align}\label{eq:cell_cochain}
\xymatrix{
\cdots&\ar[l]_-{\partial^{\ell}}  C^\ell(X) & \ar[l]_-{\partial^{\ell-1}}C^{\ell-1}(X)& \ar[l]_-{\partial^{\ell-2}} \cdots,
}
\end{align}
where $C^\ell(X)=C_\ell(X)^*$ and $\partial^{\ell-1}=\partial_\ell^*$. Then, the cellular cohomology is defined by 
$H^\ell(X)=\kernel\partial^\ell/\image\partial^{\ell-1}$. It is known that $H^\ell(X)\simeq \Hom_K(H_\ell(X),K)$.

\end{document}